\documentclass[10pt, reqno]{amsart}
\usepackage{amssymb,latexsym}
\usepackage{eucal}
\usepackage{tikz}
\usepackage{comment}

\newtheorem{Lemma}{Lemma}
\newtheorem{theorem}{Theorem}

\newcommand{\e}{\varepsilon}
\newcommand{\al}{\alpha}

\newcommand{\succs}{\succ\!\!\!\succ \, }
\newcommand{\precs}{\prec\!\!\!\prec \,   }

\newcommand{\A}{\mathcal{A}}
\newcommand{\ph}{\varphi}
\newcommand{\btt}{\theta}

\newcommand{\FF}{\mathcal{F}}

\newcommand{\wtl}{\widetilde}

\newcommand{\rr}{\mathrm{r}}

\begin{document}

\title[Intersecting free subgroups in free products]
 {Intersecting free subgroups in free products of left ordered groups}

\author{S.\ V.\ Ivanov}
\address{Department of Mathematics\\
University of Illinois \\
Urbana\\   IL 61801\\ U.S.A.}
 \email{ivanov@illinois.edu}
\thanks{Supported in part by the NSF under grant  DMS 09-01782.}

\begin{abstract}
A conjecture of Dicks and the author on rank of the intersection of factor-free subgroups in free products of groups is proved for the case of left ordered groups.
\end{abstract}
\maketitle

\section{Introduction}

Recall that the Hanna Neumann conjecture   \cite{N1}   claims that  if $F$ is a free group of rank
${\rm r}(F)$, $\bar {\rm r}(F) := \max (  {\rm r}(F)-1, 0)$ is the reduced rank of $F$, and $H_1$, $H_2$
are finitely generated subgroups of $F$, then $\bar { \rm{r} } (H_1 \cap H_2)
\le \bar{ {\rm r}} (H_1) \bar { \rm r}  (H_2)$.  It was shown by Hanna
Neumann  \cite{N1}  that  $\bar {\rm{r} } (H_1 \cap H_2) \le 2 \bar{ \rm{r}} (H_1)
\bar { \rm{r}} (H_2)$. For more discussion and results on this problem,  the reader is referred to
\cite{D}, \cite{D2}, \cite{Fr},  \cite{Min1},  \cite{N2}, \cite{St}.

More generally,  let  $\FF = \prod_{\alpha \in I}^* G_\alpha$ be the free product of some groups
$G_\al$, $\al \in I$. According to the classic Kurosh subgroup
theorem  \cite{K}, every subgroup $H$ of $\FF$
is  a free product $ F(H) * \prod^*  s_{{\al,\beta}} H_{\al,\beta} s_{{\al,\beta}}^{-1}$, where $H_{{\al,\beta}}$ is a subgroup of $G_\al$, $s_{{\al,\beta}} \in
\FF$, and $F(H)$ is a free subgroup of
$\FF$  such that, for every $s \in \FF$
and $\gamma \in I$,  it is true that $F(H) \cap
s G_\gamma s^{-1} =\{ 1 \}$. We say that $H$
is a {\em factor-free} subgroup of  $\FF$  if $H=F(H)$ in the above form of $H$, i.e., for every $s \in
\FF$ and $\gamma \in I$, we have $H \cap s
G_\gamma s^{-1} =\{ 1 \}$. Since a factor-free subgroup $H$ of
$\FF$  is free,  the  reduced
rank $\bar{ {\rm r}} (H) := \max ( {\rm r} (H)-1,0)$ of $H$ is well defined.
Let $q^*= q^*(G_\al, \al \in I)$ denote the minimum of orders $>2$ of finite subgroups of groups $G_\al$, $\al
\in I$, and  $q^* := \infty$ if there are no such subgroups. If  $q^* = \infty$, define
$ \tfrac{q^*}{q^*-2} := 1$. It was shown by Dicks and the author \cite{DIv} that if
$H_1$, $H_2$ are finitely generated factor-free subgroups of  $\FF$, then
$$
\bar \rr(H_1\cap H_2)  \le  2\tfrac{q^*}{q^*-2}
\bar \rr(H_1) \bar \rr(H_2)   .
$$

It was conjectured by Dicks and the author \cite{DIv} that if groups $G_\al$, $\al \in I$,  contain no involutions then, similarly to the
 Hanna Neumann conjecture,  the coefficient 2 could be left out  and
 \begin{equation}\label{conj}
 \bar \rr(H_1\cap H_2)  \le  \tfrac{q^*}{q^*-2}
 \bar \rr(H_1) \bar \rr(H_2)  .
 \end{equation}

 A special case of this generalization of the Hanna Neumann conjecture is established by Dicks and the author \cite{DIv2} by proving that if $H_1, H_2$ are finitely generated factor-free subgroups of the free product
 $\FF$  all of  whose  factors are groups of order 3, then, indeed, $\bar \rr(H_1\cap H_2)  \le  \tfrac{q^*}{q^*-2}
 \bar \rr(H_1) \bar \rr(H_2) = 3  \bar \rr(H_1) \bar \rr(H_2)$.
We remark that it follows from results of \cite{DIv}  that the last inequality (as well as \eqref{conj}) is sharp and may not be improved.
Here is another special case when the conjectured inequality \eqref{conj} holds true.

\begin{theorem}\label{th1}  Suppose that $G_\al$, $\al \in I$, are left (or right) ordered groups,
$\FF = \prod_{\alpha \in I}^* G_\alpha$ is their  free product, and
$H_1$, $H_2$ are finitely generated factor-free subgroups of $\FF$. Then
$$
\bar \rr(H_1\cap H_2)  \le  \tfrac{q^*}{q^*-2}
\bar \rr(H_1) \bar \rr(H_2) = \bar \rr(H_1) \bar \rr(H_2)  .
$$
Moreover, let  $S(H_1, H_2)$ denote a set of  representatives of those double cosets $H_1 t H_2$ of $\FF$, $t \in \FF$,   that have the property $H_1 \cap t H_2 t^{-1} \ne  \{ 1 \}$. Then
$$
\bar \rr(H_1, H_2) := \sum_{s \in S(H_1, H_2)} \bar \rr(H_1\cap s H_2 s^{-1})    \le   \bar \rr(H_1) \bar \rr(H_2) .
$$
\end{theorem}

We remark that  Antol\'in, Martino, and Schwabrow \cite{ABC} proved  a more general result on  Kurosh rank of the intersection of subgroups of free products of right ordered groups by utilizing the Bass--Serre theory of groups acting on trees and some ideas of Dicks \cite{D2}. Our proof of Theorem~\ref{th1} seems to be of independent interest as it uses explicit geometric construction of graphs, representing subgroups of free products, that are often more suitable for counting arguments, see \cite{I08}, \cite{I16}.

It is fairly easy to see that Theorem~\ref{th1} implies both the Hanna Neumann conjecture and the strengthened  Hanna Neumann conjecture, put forward by W. Neumann \cite{N2}, see Sect. 5. The strengthened  Hanna Neumann conjecture claims that
if $H_1$, $H_2$ are finitely generated subgroups of  a free group $F$, then
$$
 \sum_{s \in S(H_1, H_2)} \bar \rr(H_1\cap s H_2 s^{-1})    \le   \bar \rr(H_1) \bar \rr(H_2) ,
$$
where  the set $S(H_1, H_2)$ is defined as in  Theorem~\ref{th1} with $F$ in place in $\FF$.
Recall that Friedman \cite{Fr} proved the strengthened  Hanna Neumann conjecture by making use of sheaves on graphs and Mineyev \cite{Min1} gave a proof to the strengthened  Hanna Neumann conjecture by using Hilbert modules and group actions, see also Dicks's proof \cite{D2}. Similarly to Dicks  \cite{D2} and   Mineyev  \cite{Min1}, we also use the idea of group ordering and special sets of edges, however, our arguments deal directly with core graphs of subgroups of free products that are analogous to Stallings graphs \cite{St} representing subgroups of free groups.

In Sect.~2, we introduce necessary definitions and basic  terminology. In Sect.~3, we define and study strongly positive words in free products of left ordered groups. Sect.~4 contains the proof of  Theorem~\ref{th1}. In Sect.~5, we briefly look at the case of free groups.

\section{Preliminaries}

Let $G_\al$, $\al \in I$,  be  nontrivial
groups,  $\FF = \prod_{\alpha \in I}^* G_\alpha$ be their  free product, and $H$ a finitely generated factor-free subgroup of $\FF$, $H \ne \{ 1\}$.
Consider an alphabet $\A = \cup_{\al \in I} G_\al$, where $G_{\al} \cap G_{\al'} =\{ 1\}$ if
$\al \ne \al'$.

Similarly to the graph-theoretic approach of article \cite{I08}, in a simplified version suitable for finitely generated factor-free subgroups of  $\FF$, see \cite{I99}, \cite{I01}, \cite{I10}, we first define a labeled $\A$-graph $\Psi(H)$ which  geometrically represents $H$ in a fashion analogous to the way the Stallings  graph represents a subgroup of a free group, see \cite{St}.

If $\Gamma $ is a graph, $V \Gamma $ denotes the vertex set of  $\Gamma $ and
$E \Gamma $ denotes the  set of oriented edges of $\Gamma $. If $e \in E \Gamma $,  $e_-$, $e_+$ denote the initial, terminal, resp., vertices of  $e$ and $e^{-1}$ is the edge with opposite orientation, where $e^{-1} \ne  e$ for every $e \in E\Gamma$,  $(e^{-1})_- = e_+$, $(e^{-1})_+ = e_-$.
A finite {\em path} $p = e_1 \dots e_k$ in $\Gamma$ is a sequence of edges $e_i$ such that $(e_{i})_+ = (e_{i+1})_-$,  $i=1, \dots, k-1$.
Denote $p_- := (e_1)_-$, $p_+ := (e_k)_+$, and $|p| := k$, where $|p| $ is the {\em length} of $p$. We  allow
the possibility $|p| = 0$ and  $p = \{ p_- \} = \{ p_+ \}$. A finite path $p$ is called {\em closed} if~$p_- =p_+$. An {\em infinite} path $p = e_1 \dots e_k \dots$  is an infinite sequence of edges $e_i$ such that $(e_{i})_+ = (e_{i+1})_-$ for all $i=1,2, \dots$.  If $p=e_1 \dots e_k  \dots $ and $q = f_1 \dots f_\ell \dots$ are infinite paths such that $(e_1)_- = (f_1)_-$, then
$q^{-1} p := \dots  f_\ell^{-1} \dots f_1^{-1} e_1 \dots e_k  \dots  $ is a  {\em biinfinite} path.
A path $p$ is {\em reduced} if  $p$ contains no subpath of the form $e e^{-1}$, $e \in E\Gamma$.
A closed path $p = e_1 \dots e_k$  is {\em cyclically reduced} if $|p| >0$ and both $p$ and the cyclic permutation
$e_2 \dots e_k e_1$ of $p$ are reduced paths. The  {\em core} of a graph $\Gamma$,  denoted  $\mbox{core}(\Gamma)$,
is the minimal subgraph of $\Gamma$ that contains every  edge $e$ which can be included into a cyclically reduced path in $\Gamma$.

Let $\Psi$ be a  graph whose vertex set  $V \Psi$  consists of two disjoint parts  $V_P \Psi,  V_S \Psi$, so $V \Psi = V_P \Psi \cup V_S \Psi$. Vertices in  $ V_P \Psi$ are called {\em primary} and vertices in  $ V_S \Psi$ are termed {\em secondary}.
Every edge $e \in E \Psi$ connects  primary and secondary vertices, hence $\Psi$ is a bipartite graph.   $\Psi$ is called a  {\em labeled $\A$-graph}, or simply  {\em $\A$-graph} if
 $\Psi$ is equipped with a map  $\ph : E\Psi \to \A $, called  {\em labeling}, such that, for every $e \in E\Psi$, $\ph(e) \in \A = \cup_{\al \in I} G_\al$, $\ph(e^{-1}) =  \ph(e)^{-1}$  and if
$e_+ = f_+ \in V_S \Psi$, then $\ph(e), \ph(f) \in G_\al$ for the same $\al = \btt(e_+) \in I$, called the {\em type} of  the vertex $e_+ \in V_S \Psi$ and denoted $\al=\btt(e_+)$. If $e_+ \in V_S \Psi$,  define $\btt(e) :=  \btt(e_+)$ and $\btt(e^{-1})  := \btt(e_+)$.  Thus, for every   $e \in \Psi$, we have defined  an element
$\ph(e) \in \A$,  called the { label} of $e$, and  $\btt(e)   \in I$,  the type~of~$e$.

The reader familiar with  van Kampen diagrams over a free product of groups,
see \cite{LS}, will recognize that our labeling function $\ph : E\Psi \to \A $  is defined in the way analogous to labeling functions on van Kampen diagrams over free products of groups. Recall that van Kampen diagrams are planar 2-complexes whereas  graphs are 1-complexes, however, apart from this, the ideas of cancelations and edge foldings  work equally well for both diagrams and  graphs.

An $\A$-graph $\Psi$ is called {\em irreducible} if  properties (P1)--(P3) hold true:
\begin{enumerate}
\item[(P1)] If $e, f \in E \Psi$,  $e_- = f_- \in  V_P \Psi$, and $e_+ \ne f_+$, then   $\btt(e) \ne \btt(f)$.
\item[(P2)] If $e, f \in E \Psi$,  $e \ne f$, and $e_+= f_+ \in V_S \Psi$, then  $\ph(e) \ne  \ph(f)$ in $G_{\btt(e)}$.
\item[(P3)] $\Psi$ has no multiple edges, $\mbox{deg}_\Psi v >0$ for every $v \in V\Psi$,  and there is at most one vertex of degree 1 in $\Psi$ which, if exists, is primary.
    \end{enumerate}

Suppose $\Psi$ is a connected
irreducible $\A$-graph and a primary vertex $o \in V_P\Psi$ is distinguished so that deg${}_\Psi o =1$ if $\Psi$ happens to have a vertex of degree 1. Then  $o$ is called the {\em base} vertex of $\Psi = \Psi_o$.

As usual, elements of the free product  $\FF = \prod_{\alpha \in I}^* G_\alpha$
are regarded  as words over the alphabet $\A = \cup_{\al \in I} G_\al$.  A
{\em syllable} of a word $W$ over $\A$ is a maximal subword of $W$
all of whose letters belong to the same factor $G_\al$. The {\em
syllable length} $\| W \|$ of $W$  is the number of syllables
of $W$, whereas the {\em length} $|W|$ of $W$ is the number of all
letters in $W$.
A nonempty word $W$ over   $\A $ is called {\em reduced} if every
syllable of $W$ consists of a single letter.
Clearly, $|W| = \|
W \|$ if $W$ is reduced. An arbitrary nontrivial element of the
free product $\FF$  can  be uniquely
written as a reduced word. A word $W$  is called {\em cyclically reduced} if $W^2$ is reduced.
We write $U = W$ if words $U$, $W$
are equal as elements of $\FF$. The
literal (or letter-by-letter) equality of words $U$, $W$ is denoted $U \equiv W$.

The significance of irreducible $\A$-graphs for geometric interpretation of  factor-free subgroups
$H$ of $\FF$ is given in the following.

\begin{Lemma}\label{Lm1} Suppose $H$ is a finitely generated factor-free subgroup
of the free product $\FF =  \prod_{\alpha \in I}^* G_\alpha$, $H \ne \{ 1 \}$. Then there exists a finite connected
irreducible $\A$-graph $\Psi = \Psi_o(H)$, with the  base vertex $o$,  such that a reduced word $W$ over the alphabet $\A$ belongs to $H$ if and only if there is a reduced path $p$ in $\Psi_o(H)$  such that $p_- =p_+
=o$,  $\ph(p) = W$ in   $\FF$, and $| p | = 2|W|$.
\end{Lemma}

\begin{proof} The proof is based on Stallings's folding techniques and is somewhat analogous to the proof of van Kampen lemma for diagrams over free products of groups, see \cite{LS} (in fact, it is simpler because foldings need not preserve the property of being planar for the diagram). A more general approach, suitable for an arbitrary subgroup of $\FF$, is discussed in  Lemmas~1, 4 \cite{I08}.

Let  $H = \langle V_1, \dots, V_k \rangle$ be  generated by reduced   words $V_1, \dots, V_k$ $ \in \FF$. Consider a graph $\wtl \Psi$ which consists of $k$ closed paths $p_1, \dots, p_k$ such that they have a single common vertex $o = (p_i)_-$, and $|p_i| = 2|V_i|$,  $i=1, \dots, k$.
Furthermore, we distinguish  $o$ as the base vertex of $\wtl \Psi$ and call $o$   primary, the vertices adjacent to $o$ are called secondary vertices and so on. The labeling function $\ph$ on $p_i$ is defined so that $\ph(p_i) = V_i$, $i=1, \dots, k$, where $\ph(p) := \ph(e_1) \dots \ph(e_\ell)$ if $p = e_1 \dots e_\ell$ and $e_1, \dots, e_\ell \in E\wtl \Psi$.

Clearly, $\wtl \Psi=\wtl \Psi_o$ is a finite connected   $\A$-graph with the base vertex $o$ that has the following  property

\begin{itemize}
\item[(Q)] A word $W  \in \FF$ belongs to $H$  if and only if there is a
path $p$ in $\wtl \Psi_o$ such that $p_- =p_+
=o$ and  $\ph(p) = W$.
\end{itemize}

However, $\wtl \Psi_o$ need not be irreducible  and we will do
 foldings of edges in $\wtl \Psi_o$ which  preserve  property (Q) and which aim to achieve  properties (P1)--(P2).

Assume that property (P1) fails for edges $e, f$ with  $e_- = f_- \in V_P\wtl \Psi_o$ so that  $e_+ \ne f_+$ and  $\btt(e) =\btt(f)$.
Let us redefine the labels  of all edges  $e'$ with $e'_+ = e_+$ so that $\ph(e') \ph(e)^{-1}$ does not change
and $\ph(e) = \ph(f)$ in $G_{\btt(e)}$.  Now we identify the edges $e$, $f$ and vertices $e_+$, $f_+$.  Observe that this folding preserves  property (Q)  ((P2) might fail) and decreases the edge number $|E \wtl \Psi_o|$.

If property (P2) fails for edges $e, f$ and  $\ph(e) = \ph(f)$ in $G_{\btt(e)}$, then we  identify the edges $e, f$.  Note property  (Q) still holds ((P1) might fail) and the number $|E \wtl \Psi_o|$ decreases.

Suppose property (P3) fails and there are two distinct edges  $e, f$ in $\wtl \Psi_o$ such that $e_- = f_-$, $e_+ = f_+ \in V_S \wtl \Psi_o$. By property  (Q), it follows from  $H$ being factor-free  that $\ph(e) = \ph(f)$ in $G_{\btt(e)}$.
Therefore, we can identify the edges $e, f$, thus preserving  property  (Q) and decreasing the number $|E \wtl \Psi_o|$.  If property (P3) fails so that there is a vertex $v$ of degree 1, different from $o$, then we delete $v$ along with the incident edge. Clearly, property  (Q)
still holds and the number $|E \wtl \Psi_o|$ decreases.

Thus, by induction on $|E \wtl \Psi_o|$, in polynomial time of size of input, which is the total length $\sum_{i=1}^{k} |V_i|$, we can effectively construct an irreducible  $\A$-graph $\Psi_o$ with property (Q).  Other stated properties of $\Psi_o$ are straightforward.
\end{proof}

The following lemma further elaborates on the correspondence between finitely generated
factor-free subgroups of the free product  $\FF$ and irreducible $\A$-graphs.

\begin{Lemma}\label{Lm2} Let  $\Psi_o$ be a finite connected
irreducible  $\A$-graph with the base vertex  $o$,
and $H= H(\Psi_o)$ be a subgroup of
$\FF $ that consists of all words $\varphi(p)$, where $p$ is
a  path in $\Psi_o$ with $p_- = p_+ = o$. Then $H$ is a
factor-free subgroup of  $\FF$   and
$\bar \rr(H)=- \chi(\Psi_o)$, where $\chi(\Psi_o) =  |V \Psi_o | - \frac 12|E \Psi_o | $ is
the Euler characteristic of~$\Psi_o$.
\end{Lemma}

\begin{proof} This follows from  the facts that the fundamental group $\pi_1(\Psi_o, o)$ of
$\Psi_o$ at $o$ is free of rank $- \chi(\Psi_o)+1$ and that the homomorphism $\pi_1(\Psi_o, o) \to \FF$, given by $p \to \ph(p) $, where $p$ is a path with $p_-=p_+= o$, has the trivial kernel following from  properties (P1)--(P2).
\end{proof}

Suppose $H$  is a nontrivial finitely generated factor-free subgroup
of a free product $\FF = \prod_{\alpha \in I}^* G_\alpha$, and $\Psi_o = \Psi_o(H)$ is an
irreducible  $\A$-graph for $H$ as in Lemma~\ref{Lm1}.
Let $\Psi(H)$ denote the core of $\Psi_o(H)$. Clearly, $\Psi(H)$
has no vertices of degree $\le 1$ and $\Psi(H)$ is also an irreducible $\A$-graph.
It is easy to see that the graph  $\Psi_o(H)$ of $H$ can be obtained back from the core graph $\Psi(H)$  of $H$ by attaching a suitable path $p$ to  $\Psi(H)$ so that $p$ starts at a primary vertex $o$, ends in  $p_- \in V_P\Psi(H)$, and then by doing foldings of edges as in the proof of Lemma~\ref{Lm1}.

Now suppose  $H_1$, $H_2$  are  nontrivial finitely generated factor-free subgroups of $\FF$.
Consider a set $S(H_1, H_2)$ of representatives of those double cosets
$H_1 t H_2$ of $\FF$, $t \in \FF$,   that have the property $H_1 \cap t H_2 t^{-1} \ne  \{ 1 \}$.
For every $s \in S(H_1, H_2)$,  define the subgroup $K_s := H_1 \cap s H_2 s^{-1}$.
Analogously to  the case of free groups, see  \cite{N2}, \cite{St}, we now
construct a finite irreducible  $\A$-graph $\Psi(H_1, H_2)$ whose connected components are core graphs $\Psi(K_s)$, $s \in S(H_1, H_2)$.

First we define an $\A$-graph $\Psi_o'(H_1, H_2)$. The set  of primary vertices of
$\Psi_o'(H_1, H_2)$ is
$V_P \Psi_o'(H_1, H_2)   := V_P \Psi_{o_1}(H_1)\times V_P \Psi_{o_2}(H_2) .
$
Let
$$
\tau_i : V_P \Psi_o'(H_1, H_2) \to V_P \Psi_{o_i}(H_i)
$$
denote the projection map, $\tau_i((v_1, v_2)) = v_i$, $i=1,2$.

The set of secondary vertices $V_S \Psi_o'(H_1, H_2)$ of
$\Psi_o'(H_1, H_2)$ consists of equivalence classes $[u]_\al$, where  $u \in V_P \Psi_o'(H_1, H_2)$, $\al \in I$, with respect to  the minimal  equivalence relation generated by   the following relation   $\overset \al \sim$ on  $V_P \Psi_o'(H_1, H_2)$:  Define $v \overset \al \sim w$ if there are edges $e_i, f_i \in E \Psi_{o_i}(H_i)$ such that $(e_i)_- = \tau_i(v)$, $(f_i)_- = \tau_i(w)$,  $(e_i)_+ = (f_i)_+$, $i=1,2$, the edges  $e_i, f_i$ have type $\al$, and $\ph(e_1) \ph(f_1)^{-1} = \ph(e_2) \ph(f_2)^{-1}$ in $G_\al$. It is easy to see that  $\overset \al \sim$  is symmetric and transitive on distinct pairs and triples (but it could lack reflexive property).

The edges in $\Psi_o'(H_1, H_2)$  are defined so that $u \in  V_P \Psi_o'(H_1, H_2)$ and
$[v]_\al \in V_S \Psi_o'(H_1, H_2)$ are connected by an edge if and only if $u  \in [v]_\al$.

The type of a vertex $[v]_\al \in V_S \Psi_o'(H_1, H_2)$  is $\al$ and if  $e \in E\Psi_o'(H_1, H_2)$,  $e_- =u$, $e_+ = [v]_\al$, then $\ph(e) :=\ph(e_1)$, where $e_1  \in E\Psi_{o_1}(H_1)$
is an edge of type $\al$ with $(e_1)_- = \tau_1(u)$, when such an $e_1$ exists, and $\ph(e_1) :=b_\al \ne 1$, $b_\al \in G_\al$,   otherwise.

It follows from the definitions and properties (P1)--(P2) of $\Psi_{o_i}(H_i)$, $i=1,2$, that $\Psi_o'(H_1, H_2)$   is an  $\A$-graph with properties (P1)--(P2). Hence, taking the core of
$\Psi_o'(H_1, H_2)$, we obtain an irreducible  $\A$-graph which we  denote $\Psi(H_1, H_2)$.

In addition, it is not difficult to see that, when taking the connected component $\Psi_o'(H_1, H_2, o)$  of $\Psi_o'(H_1, H_2)$  that contains the vertex $o = (o_1, o_2)$ and inductively removing  from $\Psi_o'(H_1\cap H_2, o)$  vertices of degree 1 different from $o$, we obtain an irreducible  $\A$-graph $\Psi_o(H_1\cap H_2)$ with the base vertex $o$ that  corresponds to the intersection $H_1\cap H_2$ as in Lemma \ref{Lm1}.

Observe that it follows from the definitions and property (P1) for  $\Psi(H_i)$, $i=1,2$,   that, for every edge $e \in E \Psi(H_1, H_2)$ with $e_- \in V_P \Psi(H_1, H_2)$, there are unique edges
$e_i  \in E\Psi(H_i)$ such that  $\tau_i(e_-) =(e_i)_-$, $i=1,2$. Hence, by setting  $\tau_i(e) =e_i$, $\tau_i(e_+) =(e_i)_+$, $i=1,2$,   we extend  $\tau_i$ to the graph map
$$
\tau_i :  \Psi(H_1, H_2) \to  \Psi(H_i)  , \quad i=1,2  .
 $$
It follows from  definitions that $\tau_i$ is locally injective and $\tau_i$ preserves  syllables of $\ph(p)$ for every path $p$ with primary vertices $p_-, p_+$.

\begin{Lemma}\label{Lm3} Suppose $H_1$, $H_2$ are  finitely generated factor-free
subgroups of the free product $\FF$ and the set
$S(H_1, H_2)$ is not empty.  Then the connected components of the graph  $\Psi(H_1, H_2)$ are core graphs  $\Psi(H_1 \cap s H_2 s^{-1})$ of subgroups $H_1 \cap s H_2 s^{-1}$, $s \in S(H_1, H_2)$. In particular,
$$
\bar \rr(H_1, H_2) =
\sum_{s \in S(H_1, H_2) } \bar \rr(H_1 \cap s  H_2 s^{-1}) =
-\chi(\Psi(H_1, H_2) ) .
$$
\end{Lemma}

\begin{proof} As in Lemma~\ref{Lm1}, let $\Psi_{o_i}(H_i)$ be an irreducible $\A$-graph, corresponding to the subgroup $H_i$ of $\FF$, $i=1,2$, and $\Psi(H_i)$ denote the core of $\Psi_{o_i}(H_i)$.
Let $v_i \in V_P \Psi(H_i)$,  $i=1,2$, and $q(v_i)$ denote a path in
$\Psi_{o_i}(H_i)$ with $q(v_i)_- = o_i$ and  $q(v_i)_+ = v_i$. Suppose $X_i \in \FF$, $i=1,2$, and $H_1^{X_1} \cap H_2^{X_2} \ne \{ 1\}$, where  $H_i^{X_i} := X_i H_i X_i^{-1}$. Consider
an  irreducible $\A$-graph $\Psi_{u_i}( H_i^{X_i})$, $i=1,2$. Note that the core graph
$\Psi(H_1^{X_1} \cap H_2^{X_2})$ can be identified with a connected component, denoted  $\Psi_{(X_1, X_2)}(H_1, H_2)$,
of the irreducible  $\A$-graph $\Psi(H_1, H_2)$. In addition, if $w \in V_P \Psi_{(X_1, X_2)}(H_1, H_2)$, then there are paths $p_i(w)$ in $\Psi_{u_i}( H_i^{X_i})$, $i=1,2$, such that $(p_i(w))_- = u_i$, $(p_i(w))_+ = \tau_i(w)$, and $\ph(p_1(w)) = \ph(p_2(w))$.  Furthermore, it follows from the definitions that $
X_i \ph(q(\tau_i(w))) \ph(p_i(w))^{-1} \in H_i^{X_i}$,  $i=1,2$.
Therefore, there are words $V_i \in H_i$, $i=1,2$, such that  $X_i V_i  \ph(q(\tau_i(w))) =  \ph(p_i(w))$.
Since  $\ph(p_1(w)) = \ph(p_2(w))$, we further obtain
\begin{gather}\label{x12}
X_1^{-1} X_2 = V_1 \ph(q(\tau_1(w))) \ph(q(\tau_2(w)))^{-1} V_2^{-1} .
\end{gather}

Now we can draw the following conclusion. For every pair $(X_1, X_2) \in \FF \times \FF $ such that $H_1^{X_1} \cap H_2^{X_2} \ne \{ 1\}$ and a vertex
  $w \in V_P \Psi_{(X_1, X_2)}(H_1, H_2)$, there are words $V_i \in H_i$, $i=1,2$, such that the equality \eqref{x12} holds. Since the paths $q(\tau_i(w))$,  $i=1,2$,  in \eqref{x12} depend only on a connected component of  $\Psi(H_1, H_2)$, it follows from \eqref{x12}  that  if
  $(1, X), (1, Y)$ are some pairs such that $\Psi_{(1, X)}(H_1, H_2) =  \Psi_{(1, Y)}(H_1, H_2)$, then $X \in H_1 Y H_2 \subseteq \FF$.

Conversely, if $X \in H_1 Y H_2$, then the equality   $\Psi_{(1, X)}(H_1, H_2) =  \Psi_{(1, Y)}(H_1, H_2)$
is obviously true. Thus, the set $S(H_1, H_2)$ is in bijective correspondence with connected components of $\Psi(H_1, H_2)$ and, by Lemma~\ref{Lm2},  we have $\bar \rr(H_1 \cap s H_2 s^{-1}) = -\chi(\Psi_{(1, s)}(H_1, H_2))$ for every $s \in S(H_1, H_2)$. Adding up over all $s \in S(H_1, H_2)$, we arrive  to the required equality   $\bar \rr(H_1, H_2) =    -\chi(\Psi(H_1, H_2) )$.
\end{proof}

\section{Strongly positive words in free products of left ordered groups}

Recall that $G$ is called a {\em left ordered}  group
if $G$ is equipped with a total order $\le$ which is left invariant, i.e.,
for every triple $a, b, c \in G$,  the relation $a \le b$ implies $ca \le  cb$. If $G$ is left ordered, then $G$ can also be right ordered (and vice versa). Indeed, if $\le$  is a left order on $G$ then, setting $a \preceq b$ if and only if  $a^{-1} \le b^{-1}$, we obtain a right order $\preceq$ on $G$.

Let  $G_\al$, $\al \in I$,  be  nontrivial left (or right) ordered
groups and let $\FF = \prod_{\alpha \in I}^* G_\alpha$ be their  free product. Since it will be more convenient to work with left order, we assume that $G_\al$, $\al \in I$,  are left ordered. It is well known, see \cite{KopM}, and fairly easy to show   that there exists a total order $\preceq$ on  $\FF$ which extends the left orders on groups  $G_\al$, $\al \in I$,  and which turns  $\FF$ into  a left ordered  group.

A  reduced word $W \in \FF$  is called {\em positive} if $W \succ 1$. A  reduced  word $W$ is called {\em strongly positive}, denoted  $W \succs  1$,  if every nonempty suffix of $W$   is positive, i.e., if $W \equiv  W_1 W_2$ with $|W_2| >0$, then $W_2 \succ  1$. Clearly, a strongly positive word is positive.  Note if $U, W$ are strongly positive and  $U W$
is reduced, then   $UW \succs  1$.  A word $U$ is {\em (resp. strongly) negative}   if $U^{-1}$ is  (resp. strongly)  positive.

\begin{Lemma}\label{Lm4} Suppose $S$, $T$ are  strongly positive words
and the word  $S^{-1}T$ is reduced. Then $S^{-1}T$ is either strongly positive or strongly negative.
\end{Lemma}

\begin{proof} Let  $S \equiv S_1 S_2$, $T \equiv T_1 T_2$, where $|S_2|,  |T_2| >0$. Then $S_2  \succ 1$,
$T_2 \succ 1$  by $S, T  \succs  1$. Since $S^{-1}T$ is reduced, we have $S^{-1}T \ne 1$, hence  $S^{-1}T  \succ 1 $ or  $S^{-1}T  \prec 1 $.
Assume  $S^{-1}T  \succ 1 $. Then   $T \succ S= S_1S_2 $ or $S_1^{-1}T  \succ S_2 \succ 1 $, hence, in view of $T_2 \succ 1 $, all nonempty suffixes of  $S^{-1}T$ are positive. This implies  $S^{-1}T \succs  1$. If $S^{-1}T  \prec 1 $, then, switching $S$ and $T$, we can show as above
that $T^{-1}S \succs  1$, hence   $S^{-1}T $ is strongly negative.
\end{proof}

\begin{Lemma}\label{Lm5} Suppose $W$ is a reduced word. Then there exists a factorization $W \equiv U_1 U_2^{-1}$ such that
$|U_1|, |U_2| \ge 0$ and each of $U_1, U_2$ is either empty or strongly positive.
\end{Lemma}

\begin{proof} Consider a factorization $W \equiv U_1^{\e_1} \dots  U_k^{\e_k}$, where, for every $j$, $U_j \succs 1$ and $\e_j = \pm 1$,
that would be minimal with respect to $k$. Since for every letter $a$ of $W$ either $a  \succs 1$ or $a^{-1}  \succs 1$, it follows that such a factorization exists and $k \le |W|$.

Note if $\e_j = \e_{j+1}=  1$, then $U_j U_{j+1}$ is reduced  and so
 $U_j U_{j+1}  \succs 1$. Similarly, if $\e_j = \e_{j+1}=  -1$, then $U_j^{-1} U_{j+1}^{-1} $ is reduced  and so
 $U_{j+1}  U_j \succs 1$. Hence, it follows from the minimality of $k$ that $\e_j \ne \e_{j+1}$ for all $j =1, \dots, k-1$. If now  $\e_j = -1$ and $\e_{j+1}=  1$ for some $j$, then we can use Lemma~\ref{Lm4} and conclude that   either  $U_j^{-1} U_{j+1}   \succs 1 $ or
 $U_{j+1}^{-1}  U_j  \succs 1$, contrary to minimality of $k$.
 Thus, it is proven that  either $k=1$ or $k=2$ and   $\e_1 = 1$, $\e_{2}=  -1$, as required.
\end{proof}

\begin{Lemma}\label{Lm6} Suppose $W$ is a cyclically reduced word. Then there exists a factorization $W \equiv W_1 W_2$ such that the cyclic permutation $\bar W \equiv W_2 W_1$ of $W$ is either strongly positive or strongly negative.
\end{Lemma}

\begin{proof} By Lemma~\ref{Lm5},   $W \equiv U_1 U_2^{-1}$, where  $|U_1|, |U_2| \ge 0$ and $U_j \succs 1 $ if  $|U_j| > 0$, $j=1,2$.  Since $W^2$ is reduced,  $U_2^{-1} U_1 $ is reduced and, by  Lemma~\ref{Lm4}, either $ U_2^{-1} U_1 \succs 1 $ or $ U_2^{-1} U_1\precs 1 $. Hence, $\bar W \equiv U_2^{ -1} U_1 $ is a desired cyclic permutation of $W$.
\end{proof}

\section{Proving Theorem~\ref{th1}}

Let  $G_\al$, $\al \in I$, be  nontrivial
left (or right) ordered
groups and  let $\FF $ be  their  free product equipped with a {\em left} order $\preceq$.
Also, fix a total order $\le$  on the index set $I$.

Let  $\Psi$ be a finite irreducible $\A$-labeled graph, where $\A = \cup_{\al  \in I}G_\al$.  An edge $e \in E \Psi$ is called {\em maximal} if there are reduced infinite paths $p=p(e)= e_1e_2 \dots $, $q=q(e)= f_1f_2 \dots$ in $\Psi$, where $e_j, f_j \in E\Psi$,  such that  $e = e_1$, $(e_1)_- = (f_1)_-$ is primary, $\btt(e_1) > \btt(f_1)$, and,  for every $j \ge 1$, both  $\ph(e_1\dots e_{2j})  \prec 1$ and   $\ph(f_1\dots f_{2j})  \prec~1$. Note that the vertices $(e_1\dots e_{2j})_+$,  $(f_1\dots f_{2j})_+$ are primary and $q^{-1} p = \dots f_2^{-1} f_1^{-1} e_1 e_2 \dots  $ is a reduced {biinfinite} path.

\begin{Lemma}\label{Lm7} Suppose $\Psi$ is a finite  irreducible  $\A$-labeled graph whose Euler characteristic is negative, $\chi(\Psi) <0$.   Then
$\Psi$ contains a maximal edge.
\end{Lemma}

\begin{proof} Since $\chi(\Psi) <0$, $\Psi $ has a connected component $\Psi_1$ with $\chi(\Psi_1) <0$. Without loss of generality, we may assume that $\mbox{core}\, (\Psi_1) =  \Psi_1$.
It is not difficult to see from $\chi(\Psi_1) <0$ and from $\mbox{core}\, (\Psi_1) =  \Psi_1$  that, for every pair $h, h' \in E\Psi_1$, there is a reduced path $p = h \dots h'$  whose first, last edges are $h, h'$, resp..  Pick a primary vertex $o$ in  $\Psi_1$ and two distinct edges $t_1, u_1$ with $(t_1)_- = (u_1)_- = o$.  Let $t, u$ be some reduced paths such that first edges of $t$, $u$ are $t_1, u_1$, resp., and $t_+$, $u_+$ have degree $>2$. Then it follows from the above remark that there are closed paths
$r_0, s_0$ starting at   $t_+$, $u_+$, resp., such that the path  $t r_0^2 t^{-1} u s_0^2 u^{-1}$ is reduced.
Since $\Psi_1$ is irreducible and $r_0, s_0$ are reduced, it follows  $\ph(r_0) \ne 1$, $\ph(s_0) \ne 1$ in $\FF$.

Let  $r, s$ be some cyclic permutations of the closed paths $r_0, s_0$, resp., that start at some primary vertices and $R = \ph(r)$, $S = \ph(s)$ be reduced words. Clearly, $R, S$ are cyclically reduced and $|R|  = |r|/2 >1$,  $|S|  = |s|/2 >1$. By Lemma~\ref{Lm6}, there are cyclic permutations  $\bar R, \bar S$  of $R, S$, resp., such that $\bar R^{\e_r}, \bar S^{\e_s}$ are strongly positive, where $\e_r, \e_s \in \{ \pm 1\}$. Switching from $r_0$, $s_0$ to $r_0^{-1}$, $s_0^{-1}$, resp., if necessary, we may assume that $\e_r= \e_s=-1$, i.e.,
$\bar R^{-1}, \bar S^{-1} \succs 1$.
Let $\bar r$, $\bar s$ denote
 cyclic permutations of $r, s$, resp., such that $\ph( \bar r) = \bar R$, $\ph( \bar s) = \bar S$.  Also, let $\bar r =\bar r_1 \bar r_2$,  $\bar s =\bar s_1 \bar s_2$ be factorizations of   $\bar r$, $\bar s$, resp., defined by vertices $t_+$,  $u_+$, resp..

Consider two infinite paths starting at $o = t_- = u_-$ and defined as follows. Let $T = t r_0^{+\infty }$ whose prefixes are $t r_0^{k}$, $k \ge 0$,  and $U = u  s_0^{+\infty }$ whose prefixes are $u  s_0^{\ell}$, $\ell \ge 0$. It follows from the definitions that
 $T$ starts at $t_- = o$, goes along $t$ to $t_+$ and then cycles
 around $r_0$ infinitely many times, in particular, $T$ is reduced. Similarly, $U$ starts at $u_- = o$, goes along $u$ to $u_+$ and then cycles
 around $s_0$.

 Denote  $T = t_1 t_2 \dots $, where $t_j \in E \Psi_1$, and $U = u_1 u_2 \dots $, where $u_k \in E \Psi_1$. Let $T(j_1, j_2) := t_{j_1}\dots t_{j_2}$, where $j_1 \le j_2$, denote the subpath of $T$ that starts at $(t_{j_1})_-$ and ends in $(t_{j_2})_+$. It is convenient to set $T(j,j-1) := \{  (t_j)_- \} $ for all $j \ge 1$.
Similarly,  $U(j_1, j_2) := u_{j_1} \dots u_{j_2}$, where  $j_1 \le j_2$, and  $U(j,j-1) := \{  (u_j)_- \} $ if  $j \ge 1$.

Suppose $2j > |t| + |\bar r_2|$. Then $2j - |t| - |\bar r_2|> 0$. Let
$m$ be the remainder of $2j - |t| - |\bar r_2|$ when divided by
$ |r|$. Set $m_r := m$ if $m >0$ and $m_r :=  | r|$ if $m=0$.
 Note  $T(1, 2j) = T(1, 2j-m_r)  T(2j - m_r +1, 2j)$ and $\ph(  T(2j - m_r +1, 2j) ) \equiv  \bar R_3$, where   $\bar R_3$  is a prefix of
 $\ph(\bar r)  \equiv \bar R_3 \bar R_4$ of even length $m_r>0$.
 Recall $  \ph(\bar r) =  \bar R$ and $\bar R^{-1} \succs 1$,
hence $\bar R_3^{-1} \succ 1 $ and $\bar R_3 \prec 1$. Note
$\bar R_3 =  \ph(  T(2j - m_r +1, 2j) ) \prec 1$ implies, by left invariance of  the order $\preceq $, that
\begin{equation}\label{T1}
  \ph(T(1, 2j) ) \prec \ph(T(1, 2j-m_r) )    .
\end{equation}

Now suppose $2j > |u| + |\bar s_2|$.  Let
$m'$ be the remainder of $2j - |t| - |\bar s_2|>0$ when divided by
$ |s|$. Set $m_s := m'$ if $m' >0$ and $m_s :=  | s|$ if $m'=0$. Then we can derive from $\bar S^{-1} \succs 1$,  similar to \eqref{T1}, that
\begin{equation}\label{U1}
  \ph(U(1, 2j) ) \prec   \ph(U(1, 2j-m_s) )    .
\end{equation}

The comparisons \eqref{T1}--\eqref{U1} prove that a maximal element of the infinite set
\begin{equation}\label{set}
  \{   \ph(T(1, 2j) ),  \ph(U(1, 2k) )  \mid   j \ge 0, k \ge 1 \} \subseteq \FF
\end{equation}
exists and it is the maximal  element of the finite set
\begin{equation*}
\{   \ \ph(T(1, 2j) ),  \ph(U(1, 2k) )  \mid  0   \le  2j  \le |t| + |\bar r_2|,   0 <   2k  \le  |u| + |\bar s_2| \ \} \subset \FF  .
\end{equation*}

Let $\ph(Q(1, 2j_M))$, where $j_M \ge 0$, $Q \in \{ T, U \}$,  denote the maximal element of the set \eqref{set}. Note $v_M = Q(1, 2j_M)_+$ is primary. Observe that elements $\ph(T(1, 2j) )$,  $\ph(U(1, 2k) )$, $j \ge 0$, $k \ge 1$, in   \eqref{set}  are distinct and represent $\ph$-labels of subpaths of the biinfinite path $U^{-1}T$ that connect the primary vertex $o = t_-$ to all  primary vertices of $U^{-1}T$ along $U^{-1}T$ (or its inverse). If we take another primary vertex $v$ on $U^{-1}T$  and consider the set of labels of subpaths that connect $v$ to  primary vertices of $U^{-1}T$ as above, then the resulting set can be obtained from \eqref{set}   by multiplication on the left by $\ph(h(o,v))^{-1}$, where $h(o,v)= T(1, 2j_v)$ if $v = T(1, 2j_v)_+$, $2j_v \ge 0$, and  $h(o,v)= U(1, 2k_v)$ if $v = U(1, 2k_v)_+$, $2k_v > 0$. Since left multiplication preserves the order, these remarks imply that the vertex $v_M = Q(1, 2j_M)_+$ defines a factorization of the  biinfinite path $U^{-1}T = q^{-1}p $ into infinite paths $q$, $p$, where $p =e_1e_2 \dots$,  $q =f_1f_2 \dots$, $e_j, f_j$ are edges, $j \ge 1$, so that, for every $j \ge 1$, we have $\ph(e_1e_2 \dots e_{2j}) \prec 1$ and $\ph(f_1f_2 \dots f_{2j}) \prec 1$. Therefore, if $\btt(e_1) > \btt(f_1)$, then $e_1$ is a maximal edge of $\Psi$ and if $\btt(f_1) > \btt(e_1)$, then $f_1$ is  maximal in  $\Psi$.
\end{proof}

Suppose $\Gamma$ is a finite graph. A set  $D \subseteq E \Gamma$  of edges is called {\em good} (for cutting) if the graph $ \Gamma \setminus (D \cup D^{-1})$  consists of connected components whose Euler characteristics are 0. Clearly,  $\Gamma$  contains a good edge set if and only if no connected component of $\Gamma$ is a tree.

\begin{Lemma}\label{Lm8} Suppose $\Psi$ is a finite connected irreducible $\A$-graph with $\chi(\Psi) <0$.   Then the set of all maximal edges of $\Psi$ is good (for cutting).
\end{Lemma}

\begin{proof} Arguing on the contrary, assume that $D \subseteq E\Psi$ is the set of
all maximal edges of $\Psi$ and $D$ is not good. Then the graph $\Psi  \setminus (D \cup D^{-1})$ contains a connected component  $\Psi'_1$ with either $\chi(\Psi'_1) < 0$ or  $\chi(\Psi'_1) > 0$.
If  $\chi(\Psi'_1) < 0$, then the core $\Psi_1=\mbox{core}(\Psi'_1)$ of $\Psi'_1$  is a finite  irreducible $\A$-graph with  $\chi(\Psi_1) < 0$. By Lemma~\ref{Lm7},  $\Psi_1$ contains a maximal edge $e$.  However, it follows from the definition that $e$ is also maximal for $\Psi$, hence, $e \in D$. This contradiction shows that $\chi(\Psi'_1) > 0$, hence  $\Psi'_1$ is a tree which we denote $T$.

Let $C$ denote the set that consists of all $c \in E \Psi$ such that $c_+ \in VT$ and $c \not\in ET$. It follows
from the definitions that if $c \in C$ then $c \in D$ or $c^{-1} \in D$.  Since every $d \in D$
is maximal, there are infinite reduced  paths $p(d)= e_1(d) e_2(d) \dots$ and  $q(d)= f_1(d) f_2(d) \dots$ such that $e_1(d)_- = f_1(d)_- \in V_P \Psi$, $e_1(d) = d$,  $\btt(e_1(d)) >  \btt(f_1(d))$ and, for every $j \ge 1$, $\ph( e_1(d) \dots  e_{2j}(d)) \prec 1$ and  $\ph( f_1(d) \dots  f_{2j}(d)) \prec 1$.

Pick an arbitrary $c \in C$. Suppose $c_-$ is primary. Since $d_- \in V_P\Psi$ if $d \in D$  and $c $ or $c^{-1}$ is in $D$, it follows that $c \in D$. Consider a shortest path of the form $h(c) := e_1(c) \dots e_{2\ell}(c)$, $\ell \ge 1$, such that either $e_{2\ell}(c)^{-1} \in C$ or $e_{2\ell}(c)^{-1} \in T$ and $e_{2\ell+1}(c)^{-1} \in C$. Define $\sigma(c):= e_{2\ell}(c)^{-1}$ if  $e_{2\ell}(c)^{-1} \in C$ and $\sigma(c):= e_{2\ell+1}(c)^{-1}$ if $e_{2\ell+1}(c)^{-1} \in C$. Since $T$ is a finite tree, such a path $h(c)$ exists, $|h(c)| >0$,  $\sigma(c) \ne c$, and $\ph(h(c)) \prec 1$. Note $h(c)_-$, $h(c)_+$ are primary vertices of $c$, $\sigma(c)$, resp., and
$h(c) = c h_T(c) \sigma(c)^{-\e_{\sigma(c)}}$,  where $h_T(c)$ is a subpath of $h(c)$ in $T$ with
$h_T(c)_- = c_+$, $h_T(c)_+ = \sigma(c)_+$, and $\e_{\sigma(c)}= 1$ if $\sigma(c)_+$ is secondary and  $\e_{\sigma(c)}= 0$ if $\sigma(c)_+$ is primary.

 Now assume that  $c_+$ is primary. Then $c^{-1} =d_c \in D$. Consider a shortest path of the form $h(c) := f_1(d_c) \dots f_{2\ell-2}(d_c)$, where $\ell \ge 1$  and if $\ell = 1$ then $h(c) := \{ c_+ \}$, such that either $f_{2\ell-2}(d_c)^{-1} \in C$
 or $f_{2\ell-2}(d_c)^{-1} \in T$ (or $f_{2\ell-2}(d_c)$ is undefined if $\ell = 1$) and  $f_{2\ell-1}(d_c)^{-1} \in C$.  Define $\sigma(c):= f_{2\ell-2}(d_c)^{-1}$ if  $f_{2\ell-2}(d_c)^{-1} \in C$ and $\sigma(c):= f_{2\ell-1}(d_c)^{-1}$ if   $f_{2\ell-1}(d_c)^{-1} \in C$. Since $T$ is a finite tree, such a path $h(c)$ exists,  $|h(c)| \ge 0$,  $\sigma(c) \ne c$,   and $\ph(h(c)) \preceq 1$. In addition,  the equality $\ph(h(c)) = 1$ implies that $h(c) = \{ c_+ \}$,
 $\sigma(c) = f_{1}(c)^{-1}$ and $\btt(\sigma(c)) =  \btt(f_1(d_c))   <  \btt(e_1(d_c)) = \btt(c)$.
 As above, we remark that $h(c)_-$, $h(c)_+$ are primary vertices of $c$, $\sigma(c)$, resp., and
$h(c) =  h_T(c) \sigma(c)^{-\e_{\sigma(c)}}$,  where $h_T(c)$ is  in $T$ with
 $h_T(c)_- = c_+$, $h_T(c)_+ = \sigma(c)_+$, and $\e_{\sigma(c)}= 1$ if $\sigma(c)_+$ is secondary
 and  $\e_{\sigma(c)}= 0$ if $\sigma(c)_+$ is primary.

 Let us summarize. For every $c \in C$, we have defined an edge $\sigma(c) \in C$,  $\sigma(c) \ne c$,  hence,   $\sigma : C \to C$ is a function. Furthermore, there is a reduced path $h(c)$ such that $h(c) = c^{\e_c}h_T(c) \sigma(c)^{-\e_{\sigma(c)}}$,  where $h_T(c)$ is in $T$ with
 $h_T(c)_- = c_+$, $h_T(c)_+ = \sigma(c)_+$,  $\e_c =1$ if $c_+ \in V_S\Psi$ and $\e_c =0$ if $c_+  \in V_P\Psi$. In addition,  $\e_{\sigma(c)}= 1$ if $\sigma(c)_+  \in V_S\Psi$
 and  $\e_{\sigma(c)} = 0$ if $\sigma(c)_+  \in V_P\Psi$. Also,  $\ph(h(c)) \preceq 1$ and  $\ph(h(c)) = 1$ implies  that $h(c) = \{ c_+ \} = \{ \sigma(c)_+ \} $ and
 $\btt(c) > \btt(\sigma(c)) $. Finally, $h(c)_-$,   $h(c)_+$ are primary vertices of $c$, $\sigma(c)$, resp., whence  $h(c)_+ = h(\sigma(c))_-$ for every $c \in C$.

 Since $C$ is finite,  there is a cycle $c, \sigma(c), \dots, \sigma^k(c) =c$, $k \ge 2$, for some $c \in C$. Consider the closed path $p_c = h(c) h(\sigma(c)) \dots h(\sigma^{k-1}(c))$. Since
 $\ph(h(\sigma^{j}(c)))  \preceq 1$ for every $j$, we  obtain that $\ph(p_c)  \preceq 1$ and the equality $\ph(p_c) =1$ implies $\ph( h(\sigma^{j}(c)) ) =1$ and   $h(\sigma^{j}(c)) =   \{ \sigma^j(c)_+ \} $ for every $j$.
  On the other hand,
\begin{align*} p_c & = c^{\e_c} h_T(c) \sigma(c)^{-\e_{\sigma(c)}}  \sigma(c)^{\e_{\sigma(c)}}   h_T(\sigma(c)) \sigma^2(c)^{-\e_{\sigma^{2}(c)}} \ldots \\ &  \sigma^{k-1}(c)^{\e_{\sigma^{k-1}(c)}}   h_T(\sigma^{k-1}(c)) \sigma^k(c)^{-\e_{\sigma^{k}(c)}}  = c^{\e_c} h_T(c)   h_T(\sigma(c)) \ldots   h_T(\sigma^{k-1}(c))  c^{-\e_c}
 \end{align*}
following from $\sigma^{k}(c) = c$. Since $ h_T(c)   h_T(\sigma(c)) \dots   h_T(\sigma^{k-1}(c)) $ is a closed path in the tree  $T$, we have $\ph(p_c) =1$ in $\FF$. Therefore,
$h(\sigma^j(c)) = \{ \sigma^j(c)_+ \} = \{ \sigma(c)_+ \} $ and
 $\btt(\sigma^{j}(c)) > \btt(\sigma^{j+1}(c))$ for every $j =0, 1, \dots, k-1$, implying
$\btt(c) > \btt(\sigma^{k}(c)) =  \btt(c)$. This contradiction completes the proof. \end{proof}

{\em Proof of Theorem~\ref{th1}.} As in Sect. 2, consider a finite irreducible $\A$-graph $\Psi(H_1, H_2)$ whose connected components correspond to core graphs of subgroups $H_1 \cap s H_2 s^{-1}$, $s \in S(H_1, H_2)$. Without loss of generality,
we may assume that  $-\chi( \Psi(H_1, H_2) ) >0$.
Let $D$ be the set of all maximal edges in $\Psi(H_1, H_2)$. It is easy to see  from the definitions that if $d \in D$, then $\tau_i(d)$ is maximal in
$\Psi_{o_i}(H_i)$, $i=1,2$. Hence, $\tau_i(D) \subseteq D_i$, where $D_i$  is the set of  maximal edges of $\Psi_{o_i}(H_i)$, $i=1,2$. By Lemma~\ref{Lm8},  $D_i$ is good for $\Psi_{o_i}(H_i)$ and it follows from  Lemma~\ref{Lm3} and definitions that
$$
\bar \rr(H_1, H_2 )= |D| \le  |\tau_1(D)|\cdot |\tau_2(D)|  \le    |D_1|\cdot |D_2| =  \bar \rr(H_1)  \cdot \bar \rr(H_2)  ,
$$
as desired. \qed

\section{The Free Group Case}

Suppose $H_1$, $H_2$ are finitely generated subgroups of a free group $F = F(\A)$, where $\A= \{ a_1, \dots, a_m \}$ is a set of free generators of $F$. Let  $F(a, b)$ be a free group of rank 2 with free generators $a, b$.
 Note that the map $\mu : a_i \to a^i b^i a^{-i} b^{-i}$, $i =1, \dots, m$, extends to a monomorphism $\mu : F(\A) \to F(a, b)$ such  that $\mu(H_1)$, $\mu(H_2)$ are factor-free subgroups of the free product $F(a, b) = A *B$, where $A = \langle a \rangle$, $B = \langle b \rangle$ are infinite cyclic groups generated by $a, b$.  We may assume that $\mu(S(H_1, H_2)) \subseteq  S(\mu(H_1), \mu(H_2))$. Since a cyclic  group is left ordered, it follows from Theorem~\ref{th1} that
\begin{align}\notag
\bar \rr(H_1, H_2) & = \hskip-2mm \sum_{s \in S(H_1, H_2) } \hskip-2mm  \bar  \rr(H_1 \cap  s H_2{s^{-1}})  \le \hskip-4mm
 \sum_{t \in S(\mu(H_1), \mu(H_2)) } \hskip-4mm \bar  \rr(\mu(H_1) \cap t\mu(H_2){t^{-1}})  \\
  \label{SHN}  & \le \bar  \rr(\mu(H_1)) \bar \rr(\mu(H_2)) =\bar  \rr(H_1)\bar \rr(H_2)  .
\end{align}

 We remark that there is a more direct way to prove the inequality \eqref{SHN} by repeating verbatim the arguments of Sect.~4  with a few changes in basic definitions. To do this, consider a graph $U$ with $VU = \{ o_P, o_S \}$ and $EU = \{ a_1^{\pm 1}, a_2^{\pm 1},  a_3^{\pm 1} \}$, where $(a_j)_- = o_P$ and $(a_j)_+ = o_S$, $j =1,2,3$. The fundamental group $F_2 = \pi_1(U, o_P)$ of $U$ at $o_P$ is free of rank 2, and $F_2 =  \langle a_1 a_2^{-1} , a_1 a_3^{-1} \rangle$ is freely generated by $a_1 a_2^{-1}, a_1 a_3^{-1}$.  Let $H_1$, $H_2$  be finitely generated subgroups of $F_2$, $X_i$ be the Stallings graph of $H_i$, $i=1,2$, and $W$ be the core of the  pullback  $ X_1 \underset U \times X_2$   of $X_1$, $X_2$ over $U$, see \cite{St}. If $Q \in \{ X_1, X_2, W, U \}$, there is a canonical graph map $\ph_Q: Q \to U$ which is locally injective and which we call {\em labeling}. If $v \in VQ$ and $\ph_Q(v) = o_P$, $v$ is  called {\em primary}.  If $\ph_Q(v) = o_S$, $v$ is  {\em secondary}. The image  $\ph_Q(e) = a_j^{\pm 1}$ is the {\em label} of an edge $e \in EQ$ and $\btt(e) := j \in \{1,2,3\} = I$ is the {\em type} of $e$. With this terminology, the definitions and  arguments of Sect.~4 for graphs $Q$, $W$, $X_1$, $X_2$ and the group $F_2$ in place of  $\Psi$, $\Psi(H_1, H_2)$,  $\Psi_{o_1}(H_1)$, $\Psi_{o_2}(H_2)$ and $\FF$, resp.,  are retained.


\begin{thebibliography}{[10]}

\bibitem[1]{ABC} Y. Antol\'in, A. Martino, and I. Schwabrow,
{\em Kurosh rank of intersections of subgroups of free products of right-orderable groups},
preprint, \texttt{http://arxiv.org/abs/1109.0233v3}


\bibitem[2]{D}  W. Dicks,  {\em Equivalence of the strengthened
Hanna Neumann conjecture and the amalgamated graph conjecture},
Invent. Math. {\bf  117}(1994), 373--389.

\bibitem[3]{D2}  W. Dicks, {\em Simplified Mineyev}, preprint,
\texttt{http://mat.uab.cat/$\! \sim$dicks/SimplifiedMineyev.pdf}

\bibitem[4]{DIv}  W. Dicks and S. V. Ivanov, {\em On the intersection
of free subgroups in free products of groups},   Math. Proc. Cambridge
Phil. Soc. {\bf  144}(2008), 511--534.

\bibitem[5]{DIv2}   W. Dicks and S. V. Ivanov, {\em On the intersection
of free subgroups in free products of groups with no
2-torsion},   Illinois J. Math. {\bf 54}(2010), 223--248.

\bibitem[6]{Fr}
J. Friedman, {\em Sheaves on graphs, their homological invariants, and a proof of
the Hanna Neumann conjecture: with an appendix by Warren Dicks},
Mem. Amer. Math. Soc. {\bf  233}(2014), no. 1100. xii+106 pp.

\bibitem[7]{I99}  S. V. Ivanov, {\em  On the intersection of
finitely generated subgroups in free products of groups},  {
Internat. J. Algebra and Comp.}  {\bf 9}(1999), 521--528.

\bibitem[8]{I01}  S. V. Ivanov, {\em Intersecting free subgroups
in free products of groups},  { Internat. J. Algebra and Comp.}
{\bf 11}(2001), 281--290.

\bibitem[9]{I08} S. V. Ivanov,  {\em On the Kurosh rank of the
intersection of subgroups in free products of groups}, Adv.  Math. {\bf
218}(2008), 465--484.

\bibitem[10]{I10}  S. V. Ivanov, {\em A property of groups and the Cauchy–-Davenport theorem},
{J. Group Theory} {\bf 13}(2010), 21--39.

\bibitem[11]{I16}  S. V. Ivanov, {\em Linear programming and the intersection of free subgroups in free products of groups},
submitted.

\bibitem[12]{KopM}
V. M. Kopytov and N. Y. Medvedev,  {\em Right ordered groups},
Plenum Publ., New York, 1996.


\bibitem[13]{K} A. G. Kurosh, {\em The theory of groups},
Chelsea, 1956.

\bibitem[14]{LS} R. C. Lyndon and P. E. Schupp, {\em Combinatorial
group theory}, Springer-Verlag, 1977.

\bibitem[15]{Min1} I. Mineyev, {\em Submultiplicativity and the Hanna Neumann conjecture},  Ann. Math. {\bf 175}(2012), 393--414.

\bibitem[16]{N1}  H. Neumann, {\em On the intersection of finitely generated
free groups}, {  Publ. Math.} {\bf  4}(1956), 186--189; Addendum,
{ Publ. Math.} {\bf  5}(1957), 128.

\bibitem[17]{N2} W. D. Neumann, {\em On the intersection of finitely
generated subgroups of free groups},  { Lecture Notes in Math.
(Groups-Canberra 1989)}  {\bf  1456}(1990), 161--170.

\bibitem[18]{St}  J. R. Stallings, {\em Topology of finite
graphs},  {Invent. Math.} {\bf  71}(1983), 551--565.


\end{thebibliography}
\end{document}